\documentclass{article}%

\usepackage{graphicx}
\usepackage{amsmath,amssymb,amsfonts}%
\usepackage{theorem}%
\usepackage{color}%
\usepackage{hyperref}
\usepackage[font={small, it}]{caption}
\theoremstyle{change}%
\newtheorem{definition}{Definition:}[section]%
\newtheorem{proposition}[definition]{Proposition:}%
\newtheorem{theorem}[definition]{Theorem:}%
\newtheorem{lemma}[definition]{Lemma:}%
{\theorembodyfont{\rmfamily}\newtheorem{remark}[definition]{Remark:}}%
{\theorembodyfont{\rmfamily}}%

\newenvironment{proof}
{{\bf Proof:}}
{\qquad \hspace*{\fill} $\Box$}%

\newcommand{\fg}{\mathfrak{g}}%
\newcommand{\fe}{\mathfrak{e}}%
\newcommand{\tr}{\operatorname{tr}}%

\newcommand{\inner}{\operatorname{int}}%
\newcommand{\cl}{\operatorname{cl}}%
\newcommand{\rme}{\mathrm{e}}%

\newcommand{\CC}{\mathcal{C}}%
\newcommand{\EC}{\mathcal{E}}%
\newcommand{\OC}{\mathcal{O}}%
\newcommand{\UC}{\mathcal{U}}%
\newcommand{\XC}{\mathcal{X}}%
\newcommand{\N}{\mathbb{N}}%
\newcommand{\R}{\mathbb{R}}%
\newcommand{\Z}{\mathbb{Z}}%

\begin{document}

\title{Dynamics of LCSs on the group of proper motions}
\author{V\'{\i}ctor Ayala
\thanks{
Supported by Proyecto Fondecyt n${{}^\circ}$ 1190142. Conicyt, Chile.}  \\
Instituto de Alta Investigaci\'{o}n\\
Universidad de Tarapac\'{a}, Arica, Chile \and  Adriano Da Silva \\
		Departamento de Matem\'atica,\\Universidad de Tarapac\'a - Arica, Chile.
		\and
		Alejandro Otero Robles \\
		Instituto de Matem\'{a}tica\\
		Universidade Estadual de Campinas, Brazil\\
}
\date{\today }
\maketitle

\begin{abstract}
 This article describes the control behavior of any linear control systems on the group of proper motions $SE(2)$. It characterizes the controllability property and the control sets of the system. 
\end{abstract}

	

\section{Introduction}

Consider a control system on a differential manifold $M$. Any specific control determines a vector field, i.e. an ordinary differential equation on $M$. Thus, geometrically, a control system on $M$ can be seen as a family of differential equations parameterized by the set of the admissible controls.

A fundamental issue of a control system is the controllability property, which means the capacity of the system to send each initial condition to any desired state of the manifold. It happens through a finite concatenation of integral curves of the corresponding differential equations, determined by the admissible control which connects both states. And taking into account the irreversible nature of time. A more realistic property is the notion of a control set, which
is a region of $M$ where the controllability property holds in its interior. 

The classical linear control system is defined on $\mathbb{R}^n$, an Abelian Lie group. In the Geometric Control Theory context, a linear control system on a connected Lie group $G$ strongly depends on the group product, its Lie algebra $\fg$, and the exponential map connecting these structures. Furthermore, these problems naturally involves differential equations, differential topology, geometry and analysis.

This article describes the control behavior of any linear control systems on the proper motion group $SE(2)$. Precisely, it characterizes the controllability property and the control sets of the system.

As the group of all isometries of the plane preserving orientation, we consider $SE(2)$ as the  Lie group given by the semi-direct product of the circle $S^1$ and the Euclidean plane $\R^2$. After describing the general form of any linear and left-invariant vector fields on $SE(2)$, we decompose the controlled family of differential
equations induced by the system on the product $S^1\times\R^2$. It turns out that the original linear control system is equivalent to the product of a homogeneous system on the circle $S^1$ and an control-affine system $\Sigma_{\R^2}$ on the plane.

As a consequence of this decomposition, the control sets of our original system on $SE(2)$ can be obtained by lifting the respective control sets for the system on $\Sigma_{\R^2}$. Therefore, a throughfully analisys of the system $\Sigma_{\R^2}$ is made on Section 3. Since any equilibria point is contained in a control set, we start our analisys by computing the equilibria set $\EC$ of $\Sigma_{\R^2}$. Depending on the spectra of the drift, $\EC$ is an interval containing the origin or it lies on a well-determined circumference that plays a relevant role. In fact, the region determined by a second circumference, tangent to $\EC$, have invariant properties. From that, we characterize the controllability property and the control sets of the system $\Sigma_{\R^2}$.

Even though this article contains just mathematical results, we mention that the modern theory of control systems on Lie groups has relevant applications in many branches. For instance, there is a Human body representation in terms of rigid-body motion, with the vertebral column represented as a chain of 26 flexibly-coupled motion groups, \cite{Iv, Iv2}. See also \cite{Jurd} for other mechanic applications.

\section{Preliminaries}

This section introduces the basic notions of control-affine systems, control sets and linear control systems on $SE(2)$, and also to show some preliminaries results we will be using ahead. For more on the subjects, the reader should consult references \cite{DSAy1, CK, Jurd}.

	\subsection{Control-affine systems, controllability and control sets}
	
	Let $M$ be a finite dimensional smooth manifold. A \emph{control-affine system} in $M$ is determined by the family of ODE's
	\begin{flalign*}
		&&\dot{x}(s)=f_0(x(s))+\sum_{j=1}^mu_j(s)f_j(x(s)),  \;\;\;\;{\bf u}=(u_1, \ldots, u_m)\in\mathcal{U}&&\hspace{-1cm}\left(\Sigma_{M}\right)
	\end{flalign*}
	here, $f_0, f_1, \ldots, f_m$ are smooth vector fields on $M$, and $\UC$ is the set of the piecewise constant functions satisfying ${\bf u} (t)\in\Omega$ where $u\subset \R^m$ is a compact, convex subset satisfying $0\in\inner\Omega$. For any $x\in M$ and
	${\bf u}\in \UC$, the solution of $\Sigma_M$ is the unique curve $t\mapsto \varphi(t, x, {\bf u})$ on $M$ satisfying $\varphi(0, x, {\bf u})=x$. For $x\in M$ the {\it positive} and {\it negative orbits} of $\Sigma_M$ at $x$ are defined as  
	$$\mathcal{O}^+(x)=\{\varphi(t, x, {\bf u}), t\geq 0, {\bf u}\in\UC\}\;\;\mbox{ and }\;\;\mathcal{O}^-(x)=\{\varphi(-t, x, {\bf u}), t\geq 0, {\bf u}\in\UC\},$$
	respectively. We say that $\Sigma_M$ satisfies the Lie algebra rank condition (LARC) if the Lie algebra $\mathcal{L}$ generated by the vector
	fields $f_0, f_1, \ldots, f_m$, satisfies $\mathcal{L}(x)=T_{x}M$ for all $x\in M$. The system $\Sigma_M$ is said to {\it controllable} if $M=\OC^+(x)$ for all $x\in M$.

	\begin{definition}
	\label{control}
	A set $D\subset M$ is a \emph{control set} of $\Sigma_M$ if it is maximal, w.r.t. set inclusion, with the following properties: 
	
	\begin{enumerate}
		\item For any $x\in D$, there is ${\bf u}\in \mathcal{U}$ with $\varphi(\mathbb{R}_{+},x, {\bf u}) \subset D$;
		
		\item For any $x\in D$, it holds that $D\subset \overline{\mathcal{O}^{+}(x)}$.
	\end{enumerate}
	
	\end{definition}
	
		Let $N$ be another smooth manifold and  
	\begin{flalign*}
		&&\dot{y}(t)=g_0(x(t))+\sum_{j=1}^mu_j(t)g_j(y(t)),  \;\;\;\;{\bf u}\in\mathcal{U}&&\hspace{-1cm}\left(\Sigma_{N}\right)
	\end{flalign*}
	a control-affine system on $N$. If $\psi:M\rightarrow N$ is a smooth map, we say that $\Sigma_M$ and $\Sigma_N$ are $\psi$-conjugated if their  respective vector fields are $\psi$-conjugated, that is, 
	$$\psi_*f_j=g_j\circ\psi, \;\;\;\;j=0, 1,\ldots, m.$$
	If such $\psi$ exists, we say that $\Sigma_M$ and $\Sigma_N$ are conjugated. If $\psi$ is a diffeomorphism, $\Sigma_M$ and $\Sigma_N$ are called {\it equivalent}. It is straightforward that the image by $\psi$ of a control set of $\Sigma_M$ is contained in a control set of $\Sigma_N$.


\subsection{Linear control systems on $SE(2)$}

The group of proper motions $SE(2)$ is the group of all isometries of $\R^2$, which preserves orientation. As a Lie group, $SE(2)$ can be seen as the semi-direct product $SE(2)=S^1\times_{\rho}\R^2$, with product given by 
$$(t_1, v_1)*( t_2, v_2)=(t_1+t_2, v_1+\rho_{t_1}v_2),$$
where $S^1=\R/2\pi\Z$ and $\rho_{t}$ is the counter-clockwise rotation of $t$-degrees.
The Lie algebra of $SE(2)$ is given by the semi-direct product $\fe(2)=\R\times_{\theta}\R^2$ with bracket determined by the relation\footnote{This bracket is obtained by considering the left-invariant vector fields on $SE(2)$.}
$$[(1, 0), (0, \eta)]=(0, \theta\eta)\;\;\forall \eta\in\R^2, \hspace{1cm}\mbox{ where }\hspace{1cm}\theta=\left(\begin{array}{cc}
    0 & -1 \\
    1 & 0
\end{array}\right).$$
The relationship between the group product and the bracket comes from the equality $\rho_t=\rme^{t\theta}$.

Following \cite[Proposition 3.4]{DSAy1}, a {\it linear vector field} on $SE(2)$ has the following expression
$$\XC(t, v)=(0, Av+\Lambda_t\xi),$$
where $\xi\in\R^2$, $A\in \mathfrak{gl}(2, \R)$ satisfies $A\theta=\theta A$ and  
$$\Lambda_t\xi:=(1-\rho_t)\theta\xi.$$
Since a linear vector field is entirely determined by the matrix $A\in\mathfrak{gl}(2, \R)$ and the vector $\xi\in\R^2$, we will use the notation $\XC=(\xi, A)$ to represent a linear vector field.


A {\it left-invariant} vector field is given by 
$$Y^L(t, v)=(\alpha, \rho_t\eta), \;\;\;\mbox{ for some }\;\;\;(\alpha, \eta)\in\R\times\R^2.$$
We can now define linear control systems on $SE(2)$.

\begin{definition}
A (one-input) linear control system (LCS) on $SE(2)$ is the family of ODE's 
given by 
\begin{flalign*}
	  && \dot{g}=\XC(g)+uY^L(g), \;\;\;g=(t, v)\in E(2)\;\;\mbox{ and }\;\; u\in\Omega&&\hspace{-1cm}\left(\Sigma_{SE(2)}\right)
	  \end{flalign*}

with $\Omega=[u^-, u^+]$ and $u^-<0<u^+$. 
\end{definition}

By a straightforward calculation between the brackets of $\XC$ and $Y^L$, one obtains that 
\begin{equation}
\label{LARC}
\Sigma_{SE(2)} \mbox{ satisfies the LARC }\;\;\;\iff\;\;\; \alpha\neq 0 \;\;\mbox{ and }\;\; \alpha\xi+A\eta\neq 0.
\end{equation}

The next result shows that a LCS on the group of proper motions is equivalent to the product of a homogeneous system in $S^1$ by an affine system on $\R^2$.

\begin{proposition}
\label{conjugation}
Let $\Sigma_{SE(2)}$ be a LCS on $SE(2)$. Assume the associated linear vector field $\XC=(\xi, A)$ of $\Sigma_{SE(2)}$ satisfy $\det A\neq 0$. Therefore, $\Sigma_{SE(2)}$ is equivalent to a system on $S^1\times\R^2$ of the form
 \begin{flalign*}
	  &&\left\{\begin{array}{l}
     \dot{t}=u\alpha\\
     \dot{v}=(A-u\alpha\theta)v+u\eta
\end{array}\right.  &&\hspace{-1cm}\left(\Sigma_{S^1\times\R^2}\right)
	  \end{flalign*}

\end{proposition}

\begin{proof}
Let us assume that $\Sigma_{SE(2)}$ is determined by the vectors 
$$\XC(t, v)=(0, Av+\Lambda_t\xi)\;\;\;\mbox{ and }\;\;\; Y^L(t, v)=(\alpha, \rho_t\eta_1), \;\;\;\alpha\neq 0.$$
Since $\det A\neq 0$ the map 
$$\psi_1:SE(2)\rightarrow SE(2), \;\;\;\; \psi_1(t, v)=(t, v+\Lambda_tA^{-1}\xi),$$
is a well defined diffeomorphism of $E(2)$ satisfying
$$\forall (a, w)\in T_{(t, v)}SE(2), \;\;\;\;\;(d\psi_1)_{(t, v)}(a, w)=(a, a\rho_tA^{-1}\xi+w).$$
On the other hand, 
$$A\theta=\theta A\;\;\;\;\mbox{ implies that }\;\;\;\;\forall t\in\R, w\in\R^2, \;\;\; A\Lambda_t w=\Lambda_t Aw,$$
and hence,  
$$(d\psi_1)_{(t, v)}\XC(t, v)=\XC_0(\psi_1(t, v))\;\;\mbox{ and }\;\;(d\psi_1)_{(t, v)}Y^L(t, v)=Y^L_0(\psi_1(t, v)),$$
where $\XC_0(t, v):=(0, Av)$ and $Y_0^L(t, v)=(\alpha, \rho_t\eta),$ for $\eta=\alpha A^{-1}\xi+\eta_1$. It turns out that $\Sigma_{SE(2)}$ is equivalent to the LCS $\Sigma^0_{SE(2)}$ on $SE(2)$ determined by $\XC_0$ and $Y^L_0$.

Next, we consider the application
$$\psi_2:SE(2)\rightarrow SE(2), \;\;\;\;\psi_2(t, v)=(t, \rho_{-t}v).$$
It is straightforward to see that $\psi_2$ is a diffeomorphism satisfying 
$$\forall (a, w)\in T_{(t, v)}SE(2), \;\;\;\;\;(d\psi_2)_{(t, v)}(a, w)=(a, -a\theta\rho_{-t}v+\rho_{-t}w).$$
Consequently,
$$(d\psi_2)_{(t, v)}\XC_0(t, v)=\XC_0(\psi_2(t, v))\;\;\mbox{ and }\;\; (d\psi_2)_{(t, v)}Y_0^L(t, v)=Z(\psi_2(t, v)),$$
where $Z(t, v)=(\alpha, -\alpha \theta v+\eta)$.

 Therefore, $\Sigma^0_{SE(2)}$ is equivalent to the system on $S^1\times\R^2$ given by
 \begin{flalign*}
	  &&\left\{\begin{array}{l}
     \dot{t}=u\alpha \\
     \dot{v}=(A-u\alpha\theta)v+u\eta
\end{array}\right., \; u\in\Omega &&\hspace{-1cm}\left(\Sigma_{S^1\times\R^2}\right)
	  \end{flalign*}
concluding the proof.
\end{proof}

\begin{remark}
The previous proposition shows that the dynamics of a LCS on $SE(2)$ whose associated linear vector field $\XC=(\xi, A)$ satisfy $\det A\neq 0$, is the same as the dynamics of the product of a homogeneous system on $S^1$ and a particular class of control-affine systems on $\R^2$.
\end{remark}

\section{A particular class of control-affine systems on $\R^2$}

In this section, we discuss controllability and control sets of the particular class of control-affine systems on $\R^2$ which we obtained in the previous section.

Let $A, \theta\in\mathrm{gl}(2, \R)$ be the matrices given by 
$$A=\left(\begin{array}{cc}
	\lambda & -\mu \\
	\mu & \lambda
\end{array}\right)\;\;\mbox{ and }\;\;\theta=\left(\begin{array}{cc}
	0 & -1 \\
	1 & 0
\end{array}\right)\;\;\;\mbox{ with }\;\;\;\lambda^2+\mu^2\neq 0.$$

Consider the control-affine system on $\R^2$ given by 
\begin{flalign*}
	&&\dot{v}=(A-u\theta)v+u\eta,\;\;\;\;\;u\in\Omega &&\hspace{-1cm}\left(\Sigma_{\R^2}\right)
\end{flalign*}
where $\eta\in\R^2$ is a fixed nonzero vector. Defining $A(u):=A-u\theta$, the previous choices of $A$ and $\theta$ gives us 
$$\tr A(u)=\tr A= 2\lambda\hspace{1cm}\mbox{ and }\hspace{1cm}\det A(u)=\lambda^2+(\mu-u)^2,$$
implying, by our choices, that   
$$\det A(u)\neq 0 \;\;\;\iff\;\;\; \lambda\neq 0 \;\;\mbox{ or }\;\;\; u\neq\mu.$$
Moreover, if $\det A(u)\neq 0$, the solutions of $\Sigma_{\R^2}$ are builded through concatenations of the solutions for constant controls $u\in\Omega$
$$\varphi(s, v, u)=\rme^{s\lambda}R_{s(\mu-u)}(v-v(u))+v(u), \;\;\;\mbox{ where }\;\;\; v(u)=-A(u)^{-1}\eta,$$
and $R_{s(\mu-u)}$ is the rotation by $s(\mu-u)$, which is clockwise if $s(\mu-u)<0$ and counter-clockwise if $s(\mu-u)>0$.

\subsection{Equilibria of $\Sigma_{\R^2}$}

 In this section, we analyze the set $\EC$ of the equilibrium points of the system $\Sigma_{R^2}$. The importance of such a set comes from the fact that any point in $\EC$ certainly satisfies conditions 1. and 2. of Definition \ref{control} and hence, is contained in a control set of $\EC$ .

The equilibrium are the points $v\in\mathbb{R}^2$ such that 
$$A(u)v+u\eta=0, \mbox{ for some }u\in\Omega.$$
If $\det A(u)\neq 0$ the equilibria are given by 
\begin{equation}
    \label{equilibria}
    v(u)=-uA(u)^{-1}\eta=\frac{-u}{\lambda^2+(\mu-u)^2}\bigl(\lambda\eta-(\mu-u)\eta^*\bigr).
\end{equation}

Next, we describe some geometrical features of the set $\EC$.

\begin{proposition}
    For the set of equilibria $\EC$, it holds that:
    \begin{itemize}
        \item[1.] $\tr A=0$ and $\EC$ is an interval in the line $\R\cdot\eta^*$, containing the origin;
        \item[2.] $\tr A\neq 0$ and $\EC$ lies on the circumference with center $\zeta$ and radius $R$ (Figure \ref{figura2}) given, respectively, by
        $$\zeta=-\frac{1}{2}\left(\frac{\mu}{\lambda}\eta+\eta^*\right)\hspace{1cm}R=\frac{1}{2}\sqrt{\frac{\mu^2+\lambda^2}{\lambda^2}}|\eta|$$
    \end{itemize}
\end{proposition}

\begin{proof}
    1. It follows direct from equation (\ref{equilibria}).
    
    2. For any $u\in\Omega$, we have
    $$\left|v(u)-\zeta\right|^2=\left|\frac{-u}{\lambda^2+(\mu-u)^2}\bigl(\lambda\eta-(\mu-u)\eta^*\bigr)+\frac{1}{2}\left(\frac{\mu}{\lambda}\eta+\eta^*\right)\right|^2$$
    $$=\left\{\left(\frac{-u\lambda}{\lambda^2+(\mu-u)^2}+\frac{1}{2}\frac{\mu}{\lambda}\right)^2+\left(\frac{u(\mu-u)}{\lambda^2+(\mu-u)^2}+\frac{1}{2}\right)^2\right\}|\eta^2|$$
    $$=\left\{\frac{u^2\lambda^2}{(\lambda^2+(\mu-u)^2)^2}-\frac{u\mu}{\lambda^2+(\mu-u)^2}+\frac{\mu^2}{4\lambda^2}+\frac{u^2(\mu-u)^2}{(\lambda^2+(\mu-u)^2)^2}+\frac{u(\mu-u)}{\lambda^2+(\mu-u)^2}+\frac{1}{4}\right\}|\eta|^2$$
    $$=\left\{\underbrace{\frac{u^2}{\lambda^2+(\mu-u)^2}-\frac{u\mu}{\lambda^2+(\mu-u)^2}+\frac{u(\mu-u)}{\lambda^2+(\mu-u)^2}}_{=0}+\frac{\mu^2+\lambda^2}{4\lambda^2}\right\}|\eta|^2=R^2,$$
    showing item 2.
\end{proof}

\begin{remark}
Note that $-\eta^*$ also belongs to the previous circumference and that 
$$v(u)\rightarrow-\eta^*\;\;\;\mbox{ as }\;\;\;u\rightarrow\pm\infty,$$
Moreover, if $\mu\in\Omega$ we get that $v(\mu)$ and $-\eta^*$ are antipodal points.
\end{remark}

\subsection{Invariant subsets}

In this section, we construct a circumference that contains the equilibria of $\Sigma_{\R^2}$. As we show, the regions delimited by such circumference have some invariant properties which, we will use to obtain properties of the control sets of the system.

\begin{lemma}
\label{technical}
For any $\sigma, \nu\in\R^*$, the function 
$$f_{\sigma, \nu}:\R^*\rightarrow\R, \;\;\;\;f_{\sigma, \nu}(s)=\frac{1-2\rme^{s \sigma}\cos s\nu+\rme^{2\sigma}}{(1-\rme^{s\sigma})^2},$$
satisfies
$$f_{\sigma, \nu}(s)<\lim_{s\rightarrow 0}f_{\sigma, \nu}(s)=\frac{\sigma^2+\nu^2}{\sigma^2}$$
\end{lemma}

\begin{proof}
    The fact that,
    $$f_{-\sigma, \nu}(s)=f_{\sigma, \nu}(-s)\hspace{1cm}\mbox{ and }\hspace{1cm}f_{\sigma, -\nu}(s)=f_{\sigma, \nu}(s),$$
    allows us to assume w.l.o.g. that $\sigma, \nu\in\R^+$.
    Also, by simple calculations, one see that 
    $$f_{\sigma, \nu}(s)=\frac{\sigma^2+\nu^2}{\sigma^2}+\frac{F(s)}{\sigma^2(1-\rme^{s\sigma})^2},$$
    where 
    $$F(s):=2\sigma^2\rme^{s\sigma}(1-\cos s\nu)-\nu^2(1-\rme^{s\sigma})^2.$$
   Now, derivation of $F$ gives us that 
 $$F'(s)=2\sigma^3\rme^{s\sigma}(1-\cos s\nu)+2\sigma^2\nu\rme^{s\sigma}\sin s\nu+2\sigma\nu^2\rme^{s\sigma}(1-\rme^{s\sigma})$$
 $$=2\sigma\rme^{s\sigma}\left[\sigma^2(1-\cos s\nu)+\sigma\nu+\sin s\nu^2(1-\rme^{s\sigma})\right]=2\sigma\rme^{s\sigma}\left[\left(\sigma^2+\nu^2\right)-\left(\sigma^2\cos s\nu-\sigma\nu\sin s\nu+\nu^2\rme^{s\nu}\right)\right]$$
 $$=2\sigma\rme^{s\sigma}\Bigl[G(0)-G(s)\Bigr],$$
 where for simplicity we define 
 $$G(s)=\sigma^2\cos s\nu-\sigma\nu\sin s\nu+\nu^2e^{s\sigma}.$$
 
Derivation of $G$, gives us that 
$$G'(s)=-\sigma\nu\left(\nu\cos s\nu+\sigma\sin s\nu-\nu\rme^{s\sigma}\right),$$
and hence:
\begin{itemize}
    \item[1.] If $s\in(0, +\infty)$, the relations, 
    $$\sin s\nu< s\nu\hspace{1cm}\mbox{ and }\hspace{1cm} 1+s\nu-\rme^{s\nu}<0,$$ 
    are valid. Hence, 
    $$G'(s)=-\sigma\nu\left(\nu\cos s\nu+\sigma\sin s\nu-\nu\rme^{s\sigma}\right)>-\sigma\nu^2\left(\cos s\nu+(s\sigma-\rme^{s\sigma})\right)>-\sigma\nu^2\left(\cos s\nu-1\right)\geq 0,$$
    showing that $G$ is strictly increasing on the interval $(0, +\infty)$, and hence  
    $$\forall s>0, \;\;\;\;\;\;G(s)>G(0).$$

    \item[2.] Now, if $s\in (-\infty, 0)$ is a critical point of $G$.   
    \begin{equation}
    \label{derivada}
    G'(s)=0\;\;\;\;\iff\;\;\;\; \nu\cos s\nu+\sigma\sin s\nu=\nu\rme^{s\sigma}.
    \end{equation}
    Therefore, at a critical point, 
$$G(s)=\sigma^2\cos s\nu-\sigma\nu\sin s\nu+\nu^2e^{s\sigma}=\sigma^2\cos s\nu-\sigma\nu\sin s\nu+\nu\left(\nu\cos s\nu+\sigma\sin s\nu\right)$$
$$=(\sigma^2+\nu^2)\cos s\nu=G(0)\cos s\nu.$$
By continuity, for any $k\in \N$, there exists $s_k\in [-2k\pi, 0]$ such that 
$$\forall s\in [-2k\pi, 0], \;\;\;\;\; G(s)\leq G(s_k).$$
Since, $s_k\in (-2k\pi, 0)$ implies that $s_k$ is a critical point, the previous calculations imply that $$G(s_k)=G(0)\cos s_k\nu\leq G(0)\;\;\;\mbox{ when }\;\;\; s_k\in (-2k\pi, 0).$$
On the other hand, 
$$G(-2k\pi)=(\sigma^2+\nu^2\rme^{-2k\pi})<\sigma^2+\nu^2=G(0),$$
and hence $G(s)\leq G(0)$ for all $s\in [-2k\pi, 0]$. By considering $k\rightarrow+\infty$ allows us to conclude that 
$$\forall s<0, \;\;\;\;\;\;G(s)\leq G(0).$$

Now, if for some $s_0<0$ the equality $G(s_0)=G(0)$ holds, then $s_0$ is a local maximum and hence a critical point.  Therefore,  
$$G(0)=G(s_0)=G(0)\cos s_0\nu\;\;\;\implies\;\;\;s_0\nu= -2k\pi\;\;\mbox{ for some }\;\;k\in \N.$$
However, by relation \ref{derivada} it holds that
$$0=G'(s_0)=G'\left(\frac{-2k\pi}{\nu}\right)\;\;\;\iff\;\;\; \nu=\nu\rme^{-2k\pi\frac{\sigma}{\nu}}<\nu,$$
which is impossible. As a matter of fact, 
$$\forall s<0, \;\;\;\;\;\;G(s)<G(0).$$

 \end{itemize}
 
 By the previous items, we conclude that 
 $$F'(s)>0\;\;\;\mbox{ for }\;\;\;s<0\hspace{1cm}\mbox{ and }\hspace{1cm}F'(s)<0\;\;\;\mbox{ for }\;\;\;s>0$$
 and hence, $F$ is strictly increasing on $(-\infty,0)$ and strictly  decreasing on $(0, +\infty)$ and hence 
 $$\forall s\neq 0, \;\;\;\; F(s)<F(0)=0\;\;\;\;\implies\;\;\;\; f_{\sigma, \nu}(s)<\frac{\sigma^2+\nu^2}{\sigma^2}.$$
 The expression forthe limit of $f_{\sigma, \nu}$ as $s$ goes to zero, is obtained by using L'Hospital and the fact that $F'(0)=F''(0)=0$, which concludes the proof.
 \end{proof}

Assume that $\tr A=2\lambda\neq 0$ and consider the closed ball 
\begin{equation}
    \label{bola}
    B=\left\{v\in\R^2;\; |v+\eta^*|^2\leq \frac{\lambda^2+\mu^2}{\lambda^2}|\eta|^2\right\}.
\end{equation}

\begin{figure}[h!]
	\centering
	\includegraphics[scale=.5]{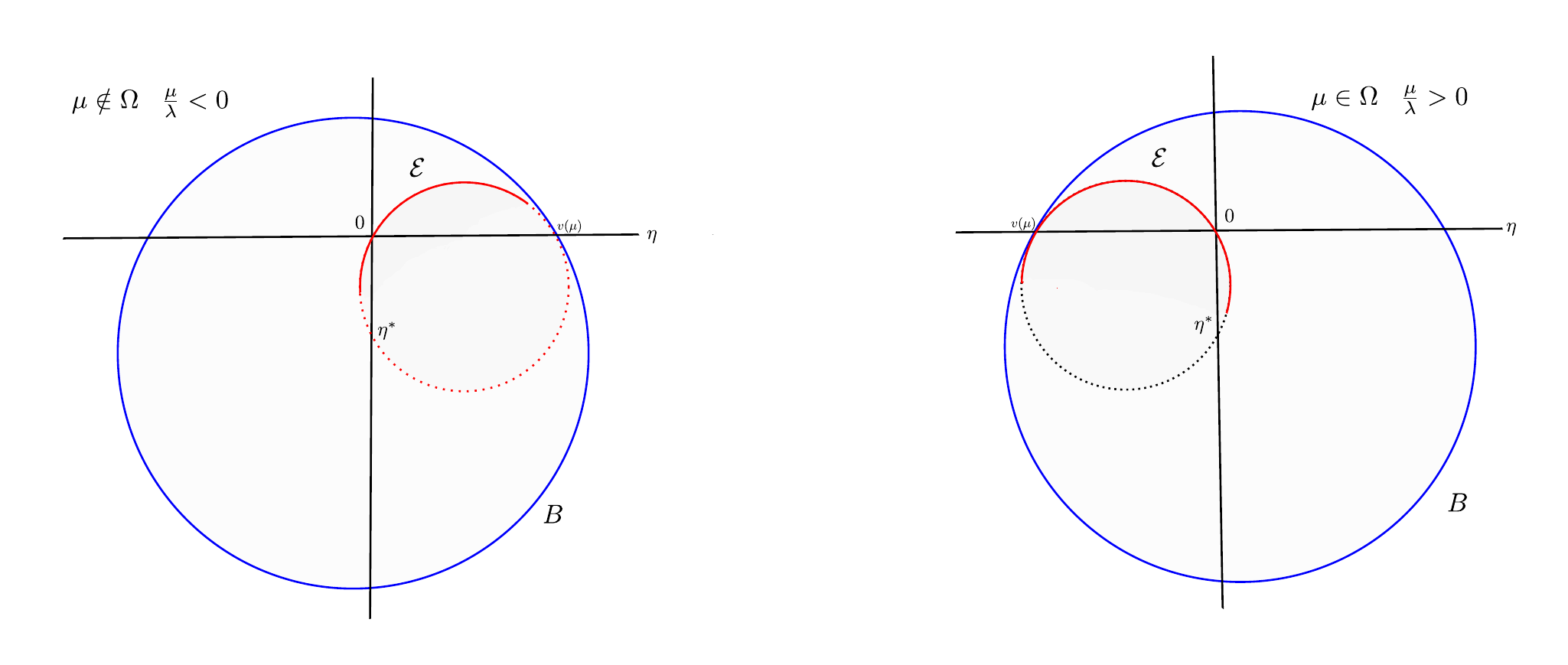}
	\caption{Equilibria of $\Sigma_{\R^2}$ and the set $B$.}
	\label{figura2}
\end{figure}

\begin{proposition}
\label{invariance}
    For all $u\in\Omega$, it holds that:
    \begin{itemize}
        \item[(i)] If $w\in B$ then $\varphi(s, w, u)\in \inner B$ for all $s\in\R$ with $\lambda s<0$;
        
        \item[(ii)] If $w\in \overline{\R^2\setminus B}$ and $w\neq v(u)$ then, for all $s\in\R$ with $\lambda s>0$,
        $$|\varphi(s, w, u)+\eta^*|> |w+\eta^*|.$$

    \end{itemize}
\end{proposition}

\begin{proof}
    For any $w\in\R^2$ and $u\in\Omega$ we have that 
    $$\left|(1-\rme^{sA(u)})w\right|^2=|w|^2-2\langle\rme^{sA(u)}w, w\rangle+\left|\rme^{sA(u)}w\right|^2$$
    $$=|w|^2-2\rme^{s\lambda}\cos(s(\mu-u))|w|^2+\rme^{2s\lambda}|w|^2=\left(1-2\rme^{s\lambda}\cos(s(\mu-u))+\rme^{2s\lambda}\right)|w|^2.$$
    On the other hand.
    $$|v(u)+\eta^*|^2=\left|\frac{-u\lambda}{\lambda^2+(\mu-u)^2}\eta+\left(\frac{u(\mu-u)}{\lambda^2+(\mu-u)^2}+1\right)\eta^*\right|^2$$
    $$=\left\{\left(\frac{-u\lambda}{\lambda^2+(\mu-u)^2}\right)^2+\left(\frac{u(\mu-u)}{\lambda^2+(\mu-u)^2}+1\right)^2\right\}|\eta|^2=\frac{\lambda^2+\mu^2}{\lambda^2+(\mu-u)^2}|\eta|^2.$$
    Therefore, by Lemma \ref{technical} we conclude that
    $$\left|(1-\rme^{sA(u)})(v(u)+\eta^*)\right|< \sqrt{\frac{\lambda^2+(\mu-u)^2}{\lambda^2}}|1-\rme^{s\lambda}|\sqrt{\frac{\lambda^2+\mu^2}{\lambda^2+(\mu-u)^2}}|\eta|=\sqrt{\frac{\lambda^2+\mu^2}{\lambda^2}}|1-\rme^{s\lambda}||\eta|$$
    
    (i) Let us consider $w\in B$, $u\in\Omega$ and $s\in\R$ with $\lambda s<0$. Then, 
    $$\left|\varphi(s, w, u)+\eta^*\right|=\left|\rme^{sA(u)}(w-v(u))+v(u)+\eta^*\right|=\left|\rme^{sA(u)}(w+\eta^*)+(1-\rme^{sA(u)})(v(u)+\eta^*)\right|$$
    $$\stackrel{\lambda s<0}{<} \rme^{s\lambda}\left|w+\eta^*\right|+\sqrt{\frac{\lambda^2+\mu^2}{\lambda^2}}(1-\rme^{s\lambda})|\eta|\stackrel{w\in B}{\leq }\rme^{s\lambda}\sqrt{\frac{\lambda^2+\mu^2}{\lambda^2}}+\sqrt{\frac{\lambda^2+\mu^2}{\lambda^2}}(1-\rme^{s\lambda})|\eta|=\sqrt{\frac{\lambda^2+\mu^2}{\lambda^2}}|\eta|,$$
    showing that $\varphi(s, w, u)\in \inner B$.
    \bigskip
    
    (ii) Analogously, if $w\in \overline{\mathbb{R}^2\setminus B}$ satisfy $w\neq v(u)$ and $s\in\R$ satisfy $\lambda s<0$, then, 
    $$\left|\varphi(s, w, u)+\eta^*\right|=\left|\rme^{sA(u)}(w-v(u))+v(u)+\eta^*\right|=\left|\rme^{sA(u)}(w+\eta^*)+(1-\rme^{sA(u)})(v(u)+\eta^*)\right|$$
    $$\stackrel{\lambda s>0}{>} \rme^{s\lambda}\left|w+\eta^*\right|-(\rme^{s\lambda}-1)\sqrt{\frac{\lambda^2+\mu^2}{\lambda^2}}|\eta|=\rme^{s\lambda}\underbrace{\left(\left|w+\eta^*\right|-\sqrt{\frac{\lambda^2+\mu^2}{\lambda^2}}|\eta|\right)}_{\geq 0, \;w\in \overline{\R^2\setminus B}}+\sqrt{\frac{\lambda^2+\mu^2}{\lambda^2}}|\eta|$$
    $$\stackrel{\lambda s>0}{\geq}\left|w+\eta^*\right|-\sqrt{\frac{\lambda^2+\mu^2}{\lambda^2}}|\eta|+\sqrt{\frac{\lambda^2+\mu^2}{\lambda^2}}|\eta|=\left|w+\eta^*\right|,$$
    concluding the proof.
    
\end{proof}

\subsection{Control sets of $\Sigma_{\R^2}$}

In this section we characterize the possibles control sets for the control-affine system we we introduce in Section 3. We start with a result concerning the positive and negative orbits for the equilibria of the system.

\begin{proposition}
	\label{open}
	For any $u\in \inner\Omega$ with $u\neq\mu$ it holds that 
	$$\OC^+(v(u))\;\;\;\mbox{ and }\;\;\;\OC^-(v(u)),$$ 
	are open sets.
\end{proposition}

\begin{proof}
	Since the proof for the positive and negative orbits are analogous, let us show only the positive case. Moreover, the fact that  $\inner\OC^+(v(u))$ is positively invariant implies that $\OC^+(v(u))$ is open if and only if $v(u)\in\inner\OC^+(v(u))$.
	
	In order to show the former, let us consider $s>0$ and define the map
	$$f:(\inner \Omega)^2\rightarrow\R^2, \;\;f(u_1, u_2):=\rme^{sA(u_2)}\left(\rme^{sA(u_1)}\left(v(u)-v(u_1)\right)+v(u_1)-v(u_2)\right)+v(u_2).$$
	We observe that,
	$$f(u, u)=v(u)\;\;\mbox{ and }\;\;f(u_1, u_2)=\varphi(s, \varphi(s, v(u), u_1), u_2)),\;\;\;\mbox{ implying that }\;\;\; f(u^2)\subset\mathcal{O}^+(v(u)).$$
	On the other hand, a simple calculation shows
	$$\frac{\partial f}{\partial u_1}(u_1, u_2)=\rme^{sA(u_2)}\left(-s\theta(v(u)-v(u_1))+(1-\rme^{sA(u_1)})v'(u_1)\right),$$
	$$\frac{\partial f}{\partial u_2}(u_1, u_2)=-s\theta\rme^{sA(u_2)}\left(\rme^{sA(u_1)}(v(u)-v(u_1))+v(u_1)-v(u_2)\right)+(1-\rme^{sA(u_2)})v'(u_2),$$
	and hence,
	$$\frac{\partial f}{\partial u_1}(u, u)=\rme^{sA(u)}(1-\rme^{sA(u)})v'(u)\hspace{1cm}\mbox{ and }\hspace{1cm}\frac{\partial f}{\partial u_2}(u, u)=(1-\rme^{sA(u)})v'(u).$$
	Also,
	$$A(u)v(u)=-u\eta\;\;\implies\;\;\;-\theta v(u)+A(u)v'(u)=-\eta\;\;\implies\;\; v'(u)=-A(u)^{-1}(\eta-\theta v(u)),$$
	which implies, 
	$$v'(u)=0\;\;\iff\;\; \eta=\theta v(u)=-uA(u)^{-1}\theta\eta\;\;\iff\;\;-u\theta\eta=A(u)\eta=A\eta-u\theta\eta\;\;\stackrel{\det A\neq 0}{\iff}\;\;\eta =0.$$
	Therefore, $v'(u)\neq 0$ for all $u\in\Omega$. Since $u\neq\mu$, by considering $s>0$ with $R_{s(\mu-u)}=\theta$, allows us to obtain that  
	$$\left\langle \frac{\partial f}{\partial u_1}(u, u), \theta \frac{\partial f}{\partial u_2}(u, u)\right\rangle=\left\langle \rme^{sA(u)}(1-\rme^{sA(u)})v'(u), \theta(1-\rme^{sA(u)})v'(u)\right\rangle$$
	$$=\left\langle \rme^{s\lambda}\theta (1-\rme^{s\lambda}\theta) v'(u), \theta(1-\rme^{s\lambda}\theta )v'(u)\right\rangle=\left\langle \rme^{s\lambda}\theta  v'(u) -\rme^{2s\lambda}\theta^2 v'(u), \theta v'(u)-\rme^{s\lambda}\theta^2v'(u)\right\rangle$$
	$$=\rme^{s\lambda} |v'(u)|^2 + 2\rme^{2s\lambda}\underbrace{\langle \theta v'(u), v'(u)\rangle}_{=0} +   \rme^{3s\lambda}|v'(u)|^2=\rme^{s\lambda}(1+\rme^{2s\lambda})|v'(u)|^2\neq 0,$$
	which implies that the set
	$$\left\{\frac{\partial f}{\partial u_1}(u, u), \frac{\partial f}{\partial u_2}(u, u)\right\}\;\;\;\mbox{ is linearly independent}.$$
	As a consequence, for any $u\in\inner \Omega$ with $u\neq\mu$, there exists an open neighborhood $(u, u)\in U\subset (\inner\Omega)^2$ such that $f(U)$ is open. Since $v(u)\in f(U)\subset\OC^+(v(u))$ we get $v(u) \in \inner\OC^+(v(u))$ which implies that $\OC^+(v(u))$ is open for any $u\in\inner \Omega$ with $u\neq\mu$.
	
\end{proof}

\bigskip

We can now prove the main result of this section.

\begin{theorem}
\label{controlaffine}
	The system $\Sigma_{\R^2}$ admits a unique bounded control set $\CC_{\R^2}$ with nonempty interior satisfying:
	\begin{enumerate}
		\item $\tr A=0$ and $\CC_{\R^2}=\R^2$;
		\item $\tr A<0$ and $\CC_{\R^2}=\overline{\OC^+(v(u))}$ for all $u\in\Omega$;
		\item $\tr A>0$ and $\CC_{\R^2}=\OC^-(v(u))$ for all $u\in\Omega$ with $u\neq\mu$.
	\end{enumerate}
\end{theorem}

\begin{proof} 	1. If $\tr A=0$, the solutions of $\Sigma_{\R^2}$ for $u\in\Omega$ with $u\neq\mu$ are given by 
	$$\varphi(s, v, u)=R_{s(\mu-u)}(v-v(u))+v(u), \;\;\;\mbox{ where }\;\;\; v(u)=-A(u)^{-1}\eta=\frac{u}{|\mu-u|}\eta^*.$$
	As a consequence,
	$$|\varphi(s, v, u)-v(u)|=|R_{s(\mu-u)}(v-v(u))|=|v-v(u)|,$$
	proving that the image of the map $s\mapsto\varphi(s, v, u)$ lies in the circumference with radius $|v-v(u)|$ and center $v(u)$ which we shall denote by $C_{u, v}$.
	Let us consider $v_0\in\R^2$ arbitrary. In order to show that $\Sigma_{\R^2}$
	is controllable, it is enough to show the existence of a periodic orbit passing through $v_0$ and $0$. Such orbit is constructed as follows:
	
	\begin{itemize}
		\item[(a)] Let $\rho>0$ and consider $\Omega_{\rho}:=[-\rho, \rho]$ such that 
		$$\mu\notin\Omega_{\rho}\;\;\;\mbox{ and }\;\;\;\Omega_{\rho}\subset\Omega,$$
		which is possible since $\mu\neq 0$.
		
		\item[(b)] The circumference $C_{\rho, v_0}$ intersects the line $\R\cdot\eta^*$ in two points. Denote by $v_1$ the point in this intersection that is close to $v(-\rho)$. In particular, $v_1=\varphi(s_1, v, \rho)$ for some $s_1>0$;
		
		\item[(c)] If $v_1\notin v(\Omega_{\rho})$, we repeat the process in the previous item for the circumference $C_{-\rho, v_1}$, obtaining a point $v_2$. 
		
		\item[(d)] Repeating the previous process, if $v_n\notin v(\Omega_{\rho})$, we obtain, in the same way, a point $v_{n+1}$ belonging to the intersection of the circumference $C_{(-1)^n\rho, v_n}$ and the line $\R\cdot\eta^*$. By induction, it is straightforward to see that the radius $R_n$ of $C_{(-1)^n\rho, v_n}$ satisfies 
		$$R_n=|v_n-v((-1)^n\rho)|=|v_0-v_{\rho}|-n\cdot\mathrm{diam}\cdot v(\Omega_{\rho}).$$
		Therefore, there exists $N\in\N$ such that $v_N\in v(\Omega_{\rho})\cap\R\cdot\eta^*$.
		
		\item[(e)] Now, since $v_N\in v(\Omega_{\rho})$ there exists, by continuity, $u_N\in\Omega_{\rho}\subset\Omega$ satisfying $|v(u_N)|=|v_N-v(u_N)|.$ The circumference $C_{u_N ,v_N}$ passes through $v_N$ and by the origin $0$. Therefore, there exists $s_N>0$ such that $\varphi(s_N, v_N, u_N)=0$ and by concatenation we get the trajectory from $v_0$ to $0$ (Figure\ref{figura1} in blue).
		
		\item[(f)] Since the trajectory from $v_0$ to $0$ is made by choosing ``half" of the circumferences of $C_{(-1)^n\rho, v_n}$, $n=0, \ldots N-1$ and $C_{u_N, v_N}$, the inverse path obtained by the complementary half circumferences, gives us the trajectory from $0$ to $v_0$ (Figure\ref{figura1} in red).
		
	\end{itemize}
	
	By the previous items, we get a periodic trajectory passing through $v_0$ and $0$. By arbitrariness of $v_0\in\R^2$ we get that $\Sigma_{\R^2}$ is controllable.
	
\begin{figure}[h!]
	\centering
	\includegraphics[scale=1.3]{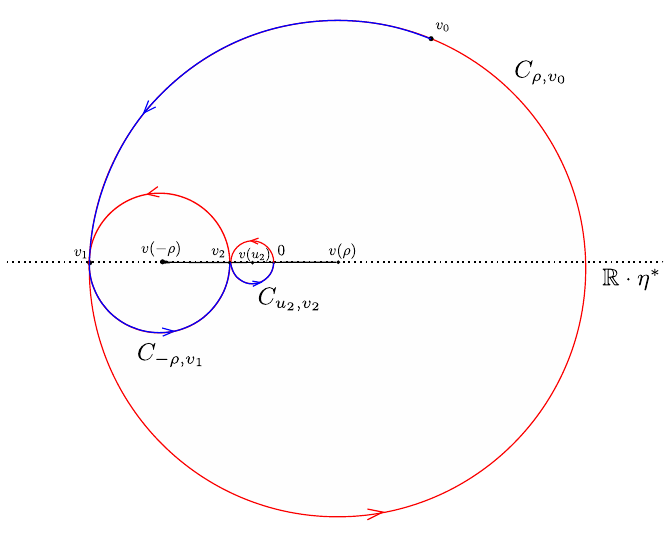}
	\caption{Periodic trajectory passing through $0$ and $v_0$}
	\label{figura1}
\end{figure}

	\bigskip

	2. Since $\tr A\neq 0$ it follows that $\det A(u)\neq 0$ for all $u\in\Omega$ implying that the solutions of $\Sigma_{\R^2}$ are given by concatenations of the curves
	$$\varphi(s, v, u)=\rme^{s\lambda}R_{s(\mu-u)}(v-v(u))+v(u).$$
	Let us analyze the case where $\lambda>0$ and $\mu>0$ since any other choices are treated similarly. For this case, let us show that $\mathcal{O}^-(v(u))$ is a control set for any $u\in\inner \Omega$ with $u\neq\mu$.
	
	By our choices, for any $u\in\Omega$, it holds that
	\begin{equation}
		\label{convergence}
		\forall v\in\R^2, \;\;\;\;\;\varphi(s, v, u)\rightarrow v(u)\;\;\mbox{ as }\;\;\;s\rightarrow-\infty.
	\end{equation}
	
	Moreover, it turns out that:
	
	\begin{itemize}
		\item[(a)] If $v\in\OC^-(v(u))$, there exists $\tau>0$ and ${\bf u}\in\UC$ such that $\varphi(\tau, v, {\bf u})=v(u)$. As a consequence, 
		$$\varphi(s, \varphi(\tau, v, {\bf u}), u)=\varphi(s, v(u), u)=v(u)\in\OC^-(v(u)),$$
		showing that $\OC^-(v(u))$ satisfies condition 1. in Definition \ref{control};
		
		\item[(b)] Let $v_1, v_2\in\OC^-(v(u))$.  By definition, there exists $\tau_1>0$ and ${\bf u}_1\in\UC$ such that $\varphi(\tau_1, v_1, {\bf u}_1)=v(u)$. On the other hand, equation (\ref{convergence}) and Proposition \ref{open} imply the existence of $\tau>0$ such that
		$$\varphi(-\tau, v_2, u)\in\OC^+(v(u))\;\;\;\implies\;\;\;\; v_2\in \OC^+(v(u)).$$
		As a consequence, there exists $\tau_2>0$ and ${\bf u}_2\in\UC$ such that $\varphi(\tau_2, v(u), {\bf u}_2)=v_2$ and by concatenation we conclude that  $v_2\in\OC^+(v_1)$. By the arbitrariness in the choices of $v_1, v_2\in\OC^-(v(u))$ we conclude that 
		$$\forall v\in\OC^-(v(u)), \;\;\;\;\; \OC^-(v(u))\subset \OC^+(v)),$$
		showing that $\OC^-(v(u))$ satisfies condition 2 in Definition \ref{control}.
		
		\item[(c)] Let $v\in\R^2$ and assume that $\OC^-(v(u))\cup\{v\}$ satisfies conditions 1. and 2. of Definition \ref{control}. By Proposition \ref{open},
		$$v(u)\in\overline{\OC^+(v)}\;\;\implies\;\;\OC^-(v(u))\cap\OC^+(v)\neq\emptyset$$
		$$\implies\;\;\exists \tau>0, {\bf u}
		\in\UC; \;\;\;\;\varphi(\tau, v, {\bf u})\in\OC^-(v(u))\;\;\;\implies\;\;\;v\in\OC^-(v(u)),$$
		proving that $\OC^-(v(u))$ is maximal, and hence, a control set of $\Sigma_{\R^2}$.
		
		\item[(d)] For any $u_1, u_2\in\Omega\setminus\{\mu\}$,  
	$$\varphi(-s, v(u_1), u_2)\rightarrow v(u_2)\;\;\;\mbox{ and }\;\;\; \varphi(-s, v(u_2), u_1)\rightarrow v(u_1), \;\;\;s\rightarrow +\infty.$$
	Thus, by Proposition \ref{open} we obtain
	$$v(u_2)\in\OC^-(v(u_1))\;\;\;\mbox{ and }\;\;\;v(u_1)\in\OC^-(v(u_2)),$$
	showing that $\CC_{\R^2}=\OC^-(v(u))$ for any $u\in\inner\Omega$ with $u\neq \mu$.
	\end{itemize}
	
	Assume now that $u^+\neq\mu$ with $\mu-u^+>0$ and let $u\in\inner\Omega$ with $\mu-u>0$. Note that the curve $s\in\R^+\mapsto\varphi(-s, v(u), u^+)$ revolves around the point $v(u)$ counter-clockwise with $\varphi(-s, v(u), u^+)\rightarrow v(u)$ as $s\rightarrow+\infty$, and the curve $\tau\in\R^+\mapsto\varphi(\tau, v(u), u^+)$ revolves around the point $v(u^+)$ clockwise with $\varphi(\tau, v(u), u^+)\rightarrow +\infty$ as $\tau\rightarrow+\infty$. Therefore, there exists $s,\tau\in\R^+$ such that 
	$$\varphi(-s, v(u), u^+)=\varphi(\tau, v(u^+), u)\;\;\;\implies\;\;\; v(u^+)\in\OC^-(v(u))=\CC_{\R^2},$$
	hence, by invariance we get $\CC_{\R^2}=\OC^-(v(u^+))$. Analogously, one argues that $\CC_{\R^2}=\OC^-(v(u^-))$ when $u^-\neq\mu$ 
	showing that $\CC_{\R^2}=\OC^-(v(u))$ for all $u\in\Omega$ with $u\neq\mu$.
	
	Next, we consider the ball $B$ given in (\ref{bola}). Since $v(u)\in B$ for all $u\in\Omega$, our assumption that $\tr A=2\lambda>0$ implies, by Proposition \ref{invariance}, that 
    	$$\OC^-(v(u))\subset B\;\;\;\implies\;\;\;\CC_{\R^2}\;\mbox{ is bounded.}$$
	
	To finish the proof, let us show the uniqueness of $\CC_{\R^2}$.  Assume that  
	$$|\varphi(s, v, u)-\CC_{\R^2}|\leq |v-\CC_{\R^2}|,$$
	 for some $s>0$, $u\in\Omega$ and $v\in \R^2$. By compactness, there exists $w\in \overline{\CC_{\R^2}}$ such that $|\varphi(s, v, u)-\CC_{\R^2}|=|\varphi(s, v, u)-w|$ and by invariance, 
	 $\varphi(-s, w, u)\in\overline{\CC_{\R^2}}$. Hence, 
	 $$|v-\CC_{\R^2}|=|\varphi(-s, \varphi(s, v, u), u)-\CC_{\R^2}|\leq |\varphi(-s, \varphi(s, v, u), u)-\varphi(-s, w, u)|$$
	 $$=\rme^{-s\lambda}|\varphi(s, v, u)-w|=\rme^{-s\lambda}|\varphi(s, v, u)-\CC_{\R^2}|\leq \rme^{-s\lambda}|v-\CC_{\R^2}|,$$
	 which gives us that 
	 $$(1-\rme^{-s\lambda})|v-\CC_{\R^2}|\leq 0\;\;\;\stackrel{\lambda>0}{\implies}\;\;\;|v-\CC_{\R^2}|=0\;\;\;\iff\;\;\;v\in\overline{\CC_{\R^2}}.$$
	 
	  Consequently, any point $v\in\R^2\setminus\overline{\CC_{\R^2}}$ satisfies, 
	  $$\OC^+(v)\subset \R^2\setminus N_{\epsilon}\left(\CC_{\R^2}\right),$$
	  where $\epsilon=|v-\CC_{\R^2}|>0$ and
	$$N_{\epsilon}\left(\CC_{\R^2}\right)=\left\{v\in\R^2;\;|v-\CC_{\R^2}|<\epsilon\right\},$$
	 is the $\epsilon$-neighborhood of $\CC_{\R^2}$, showing that one cannot reach the interior of $N_{\epsilon}\left(\CC_{\R^2}\right)$ from $\R^2\setminus N_{\epsilon}\left(\CC_{\R^2}\right)$. By condition 1. in Definition \ref{control} and the fact that, for any $v\neq v(u)$
 $$\varphi(s, v, u)\rightarrow+\infty, \;\;\;s\rightarrow+\infty,$$
	  we conclude that $\Sigma_{\R^2}$ does not admits control sets in $\R^2\setminus\overline{\CC_{\R^2}}$, implying that $\CC_{\R^2}$ is in fact the only control set with nonempty interior, which concludes the proof.
\end{proof}

\bigskip

    The proof of the uniqueness of $\CC_{\R^2}$ actually shows that, in the open case, control sets with empty interior could exists in $\partial\CC_{\R^2}$. The next result shows this is in fact the case if $\mu\in\Omega.$

	\begin{proposition}
	\label{singleton}
	If the control set $\CC_{\R^2}$ of $\Sigma_{\R^2}$ is a proper open subset of $\R^2$ and $\mu\in\Omega$, the singleton $\{v(\mu)\}$ is a control set of $\Sigma_{\R^2}$.
	\end{proposition}
	
	\begin{proof}
	    If $\mu\in\Omega$, the point $v(\mu)$ is an equilibrium of $\Sigma_{\R^2}$ and hence, there exists a control set $\CC$ with $v(\mu)\in\CC$. If we consider the ball $B$ given in (\ref{bola}), Proposition \ref{invariance} item 1. implies that $\CC_{\R^2}\subset\inner B$ implying necessarily that $$\inner \CC=\emptyset \;\;\;\mbox{ and }\;\;\;\CC\subset\partial\CC_{\R^2}.$$
	    On the other hand, item 2. in Proposition \ref{invariance} implies that 
	    $$|\varphi(s, v(\mu), u)+\eta^*|>|v(\mu)+\eta^*|\;\;\;\mbox{ if }\;\;\;u\neq\mu.$$
	    By concatenation, one concludes that  
	    $$\overline{\OC^+(v(\mu))}\cap B=\{v(\mu)\}\;\;\;\implies\;\;\;\overline{\OC^+(v(\mu))}\cap \partial\CC_{\R^2}=\{v(\mu)\},$$
	    and hence $\CC=\{v(\mu)\}$, ending the proof.
	\end{proof}

	\begin{remark}
	As shown in \cite{DSAy2}, for LCSs on $\R^2$, the whole boundary of the open control set is also a periodic orbit (and hence, a control set). For the control-affine case in this section, the previous result shows that, if $\mu\in\Omega$, the boundary of $\partial\CC_{\R^2}$ cannot by a periodic orbit. However, what happens before $\mu$ enters the control range $\Omega$ is still unknown.
	\end{remark}

\begin{remark}
It is important to observe here that Theorem \ref{controlaffine} and other results concerning control sets were proved in a general context in \cite{CK2}. However, the proof for our particular case is done by a purely geometrical approach. Moreover, such an approach allows us to show that the open control set admits an equilibrium as control set in its boundary.
\end{remark}

\section{Controllability and control sets of LCSs on $SE(2)$}

This section is devoted to stating and proving our main result concerning the control sets of a LCSs on $SE(2)$.

\begin{theorem}
Let $\Sigma_{SE(2)}$ be a LCS on $SE(2)$ with linear vector field $\XC=(\xi, A)$. Assume that $\Sigma_{\R^2}$ satisfies the LARC. It holds:  
\begin{enumerate}
     \item If $\det A=0$ then $\Sigma_{SE(2)}$ admits an infinite numbers of control sets with empty interior.

    \item If $\det A\neq 0$ we have that:
    \subitem 2.1. $\tr A=0$ and the system is controllable;
    \subitem 2.2. $\tr A<0$ and the system admits a unique compact control set $\CC_{SE(2)}$ with nonempty interior;
    \subitem 2.3. $\tr A>0$ and the system admits a unique open control set $\CC_{SE(2)}$ with a nonempty interior. Moreover, if $\mu\in\Omega$ the control set $\CC_{SE(2)}$ admits a periodic orbit in its boundary, which is a control set if $\mu\neq 0$, or $\mu=0$ and any singularity of the drift is a one-point control set on the boundary, of $\CC_{SE(2)}$.
    
\end{enumerate}
\end{theorem}

The proof of our main result will be divided into the following two sections.

\subsection{The case $\det A=0$:}

Since $\det A=0$ and $A\theta=\theta A$ we necessarily have that $A=0$. Therefore, under the LARC (see relation (\ref{LARC})) $\Sigma_{SE(2)}$ is determined by the vectors 
$$\XC(t, v)=(0, \Lambda_t\xi)\;\;\;\mbox{ and }Y^L(t, v)=(\alpha, \rho_t\eta), \;\;\;\alpha\cdot\xi\neq 0,$$
implying that the set of singularities of $\XC$ is $\{0\}\times\R^2$.

The automorphism 
$$\psi:SE(2)\rightarrow SE(2), \;\;\;\;\psi(t, v)=(t, v-\alpha^{-1}\Lambda_t\eta),$$
has a differential given by 
$$(d\psi)_{(t, v)}(a, w)=(a, w-a\alpha^{-1}\rho_t\eta),$$
which satisfies
$$(d\psi)_{(t, v)}\XC(t, v)=\XC(\psi(t, v))\;\;\;\mbox{ and }\;\;\;(d\psi)_{(t, v)}Y^L(t, v)=Z^L(\psi(t, v))=(\alpha, 0).$$
Since $\psi(0, v)=(0, v)$, under the LARC, we can assume w.l.o.g. that $\alpha=1$ and $Y^L(t, v)=(1, 0)$.

From these assumptions, the LCS is given in coordinates as

$$\left\{\begin{array}{l}
     \dot{t}=u\\
     \dot{v}=\Lambda_t \xi
\end{array}\right.$$

In particular, it is not hard to see that 
$$\varphi(s, (0, v), 0)= (0, v), \;\;\;\forall s\in\R,$$
implying that any point in $\{0\}\times\R^2$ is contained in a control set of $\Sigma_{SE(2)}$. In what follows, we show that the control set containing the point $(0, v)$ is contained in the affine space $\{0\}\times (v+\R\cdot\xi)$, which proves that $\Sigma_{SE(2)}$ admits an infinite number of control sets.

In order to show our claim, let us notice that, since for all $t\in S^1$, $\|\rho_t\|=1$, Schwarz's inequality implies that
\begin{equation}
\label{schwarz}
\langle\Lambda_t \xi, \theta\xi \rangle=\langle(1-\rho_t)\theta\xi, \theta\xi\rangle=|\xi|^2-\langle\rho_t\xi, \xi\rangle\geq  |\xi|^2-\|\rho_t\||\xi|^2= 0,
\end{equation}
where equality holds if and only if $t=0$. Moreover, by the uniqueness property, it is straightforward to see that
\begin{equation}
\label{eq3}
\forall s\in\R, t\in S^1, v, w\in\R^2\;\mbox{ and }\;{\bf u}\in\UC\;\;\;\;\;\varphi_2(s, (t, v+w), {\bf u})=\varphi_2(s, (t, v), {\bf u})+w,
\end{equation}
where $\varphi_2$ is the second coordinate function of $\varphi$.

For a given $g_0=(t_0, v_0)\in SE(2)$, let us define the set
$$K_{g_0}^+=\{(t, v)\in SE(2); \langle v-v_0, \theta\xi\rangle\geq 0\}.$$

Taking into consideration equation (\ref{schwarz}), for any ${\bf u}\in\UC$ and $(t, v)\in SE(2)$ with ${\bf u}\not\equiv 0$, the function
$$s\in\R\mapsto H(s):=\langle\varphi_2(s, (t, v), {\bf u}), \theta\xi\rangle,$$
is increasing. As a consequence, if $(t, v)\in K_{g_0}^+$, we get 
$$\langle\varphi_2(s, (t, v), {\bf u})-v_0, \theta\xi\rangle\stackrel{(\ref{eq3})}{=}\langle\varphi_2(s, (t, v-v_0), {\bf u}), \theta\xi\rangle\geq  \langle \varphi_2(0, (t, v-v_0), {\bf u}), \theta\xi\rangle=\langle v-v_0, \theta\xi\rangle\geq 0,$$
for all $s>0$. Therefore, 
$$\forall s>0, \;\;\;\varphi(s, K_{g_0}^+, {\bf u})\subset K_{g_0}^+, \;\;\;\mbox{ and hence }\;\; \overline{\mathcal{O}^+(g_0)}\subset K_{g_0}^+.$$

Next, let us assume that $\CC$ is a control set containing $g_0=(t_0, v_0)$ and let $g_1=(t_1, v_1)\in\CC$ with $v_0\neq v_1$. By approximate controllability 
$$g_0\in\cl(\mathcal{O}^+(g_1))\subset K^+_{g_1}\;\implies\; \langle v_0-v_1, \theta\xi\rangle\geq 0,$$
and  
$$g_1\in\cl(\mathcal{O}^+(g_0))\subset K^+_{g_0}\;\implies\; \langle v_1-v_0, \theta\xi\rangle\geq 0,$$
implying that $\langle v_0-v_1, \theta\xi\rangle=0$ and hence $v_1\in v_0+\R\cdot\xi$. Therefore,
$$\CC\subset S^1\times \left(v_0+\R\cdot\xi\right).$$ 

On the other hand, by the condition one in the definition of control sets, for any $g=(t, v)\in\CC$ there exists ${\bf u}\in\UC$ such that $\varphi(\R^+, g, {\bf u})\subset\CC$ and by the previous discussion, we conclude that
$$\forall s_1, s_2\geq 0, \;\;\;\;\;\varphi_2(s_1, g, {\bf u})-\varphi_2(s_2, g, {\bf u})\in \R\cdot\xi.$$
As a consequence, 
$$\R\cdot\xi \ni\frac{d}{ds}\varphi_2(s, g, {\bf u})=\Lambda_{\varphi_1(s, g, {\bf u})}\xi \;\;\;\implies\;\;\;\forall s>0, \;\;\;\varphi_1(s, g, {\bf u})=0\;\;\implies\;\;\;t=0\;\mbox{ and }\;u=0,$$
showing that 
$$g=(0, v)\;\;\;\;\implies\;\;\;\;\CC\subset \{0\}\times \left(v_0+\R\cdot\xi\right),$$
 which concludes the proof.

\subsection{The case $\det A\neq 0$}

By the LARC assumption, w.l.o.g. we consider $\alpha=1$. Therefore, by Proposition \ref{conjugation}, the system is equivalent to 
\begin{flalign*}
	&&\left\{\begin{array}{l}
     \dot{t}=u\\
     \dot{v}=A(u)v+u\eta
\end{array}\right.,\;\;\;\;\;u\in\Omega &&\hspace{-1cm}\left(\Sigma_{S^1\times\R^2}\right)
\end{flalign*}
where $A(u)=A-u\theta$. Moreover, for any $u\in\Omega$, the solutions starting at $g=(t, v)\in SE(2)$ are given, in coordinates, by
$$
\varphi_1(s, g, u)=t+s u\;\;\;\mbox{ and }\;\;\;\varphi_2(s, g, u)=\rme^{sA(u)}(v-v(u))+v(u),
$$
where $v(u):=-u A(u)^{-1}\eta$ and the sum in the function $\varphi_1$ is module $2\pi$. By induction, one easily obtain that 
\begin{equation}
    \label{eq2}
    \forall t_1, t_2\in S^1, v\in\R^2, \;\;\;\;
    \varphi(s, (t_1+t_2, v), u)=(t_1, 0)+\varphi(s, (t_2, v), u),
\end{equation}
where the sum in $S^1$ is modulo $2\pi$. Now, the fact that $A$ commutes with $\theta$ implies that $A=\left(\begin{array}{cc} \lambda & -\mu \\ \mu & \lambda\end{array}\right)$. By Theorem \ref{controlaffine}, the associated system 
\begin{flalign*}
	&&\dot{v}=(A-u\theta)v+u\eta,\;\;\;\;\;u\in\Omega &&\hspace{-1cm}\left(\Sigma_{\R^2}\right)
\end{flalign*}
admits a unique control set $\CC_{\R^2}$ with nonempty interior, which is positively invariant if 
$\tr A\leq 0$ and negatively invariant if $\tr A\geq 0$\footnote{By Theorem \ref{controlaffine}, if $\tr A=0$ we have that $\CC_{\R^2}=\R^2$ which is certainly invariant in positive and negative times.}. As a consequence, the same property holds for $S^1\times \CC_{\R^2}$, implying that 
$$\mathcal{O}^+(g)\subset S^1\times \CC_{\R^2}, \hspace{2cm}\forall g\in S^1\times \CC_{\R^2}\;\;\mbox{ if }\;\;\;\tr A\leq 0$$
and 
$$\mathcal{O}^-(g)\subset S^1\times \CC_{\R^2}, \hspace{2cm}\forall g\in S^1\times \CC_{\R^2}\;\;\mbox{ if }\;\;\;\tr A\geq 0.$$

Now, a solution connecting two elements $g_1=(t_1, v_1)$ and $ g_2=(t_2, v_2)$ in $S^1\times\inner\CC_{\R^2}=\inner(S^1\times\CC_{\R^2})$ is constructed as follows:

\begin{itemize}

\item[(a)]Let $u\in\Omega$ with $v(u)\in\inner\CC_{\R^2}$. Since controllability holds on $\inner\CC_{\R^2}$, there exists ${\bf u}_1, {\bf u}_2\in\UC$, $s_1, s_2\in\R^+$ such that 
$$\varphi(s_1, g_1, {\bf u}_1)=(t_1^*, v(u))\;\;\;\;\mbox{ and }\;\;\;\;\varphi(s_2, (0, v(u)), {\bf u}_2)=(t_2^*, v_2),$$
for some $t_1^*, t_2^*\in S^1$.

\item[(b)] Since $\varphi(s, (t, v(u)), u)=(t+us, v(u))$, there exists $s_0>0$ such that 
$$\varphi(s_0, (t_1^*, v(u)), u)=(t_2-t_2^*, v(u)).$$

\item[(c)] By concatenation we get
$$\varphi(s_2, \varphi(s_0, \varphi(s_1, g_1, {\bf u}_1), u), {\bf u}_2)=\varphi(s_2, \varphi(s_0, (t_1^*, v(u)), u), {\bf u}_2)=\varphi(s_2, (t_2-t_2^*, v(u)), {\bf u}_2)$$
$$\stackrel{(\ref{eq2})}{=}(t_2-t_2^*, 0)+\varphi(s_2, (0, v(u)), {\bf u}_2)=(t_2-t_2^*, 0)+(t_2^*, v_2)=(t_2, v_2).$$
\end{itemize}

Therefore, 
$$\forall g\in \inner(S^1\times\CC_{\R^2}), \;\;\;\;\OC(g)=\inner(S^1\times\CC_{\R^2}).$$
Proocedings in the same way as in the proof of Theorem \ref{controlaffine}, allows us to conclude that $S^1\times\CC_{\R^2}$ is a control set of $\Sigma_{S^1\times\R^2}$. As a consequence, $S^1\times\CC_{\R^2}$ coincides with $SE(2)$ if $\tr A=0$, it is bounded when $\tr A\neq 0$, closed if $\tr A<0$ and open if $\tr A>0$. Moreover, the fact that the projection
$$\pi_2:S^1\times\R^2\rightarrow \R^2\hspace{1cm} \pi_2(t, v)=v,$$
conjugates $\Sigma_{S^1\times\R^2}$ and $\Sigma_{\R^2}$ implies that $S^1\times\CC_{\R^2}$ is the only control set of $\Sigma_{S^1\times\R^2}$ with nonempty interior.

It remains to show the assertion concerning the control sets on the boundary of $S^1\times\CC_{\R^2}$ when $\tr A>0$ and $\mu\in\Omega$.

\begin{itemize}

\item[1.] If $\mu\neq 0$, we obtain 
$$\varphi\left(s+\frac{2k\pi}{\mu}, (t, v(\mu)), \mu\right)=(t+(s\mu+2k\pi), v(\mu))=(t+s\mu, v(\mu))=\varphi(s, (t, v(\mu)), \mu),$$
showing that $S^1\times\{v(\mu)\}$ is a periodic orbit. In particular, $S^1\times\{v(\mu)\}$ is contained in a control set of $\Sigma_{S^1\times\R^2}$. On the other hand, the singleton $\{v(\mu)\}$ is, by Proposition \ref{singleton},  a control set of $\Sigma_{\R^2}$. Therefore, the fact that  $\pi_2(S^1\times\{v(\mu)\})=\{v(\mu)\}$ implies necessarily that $S^1\times\{v(\mu)\}$ is a control set of $\Sigma_{S^1\times\R^2}$.

\item[2.] If $\mu=0$ it follows that $\varphi(s, (t, 0), 0)=(t, 0)$ for any $t\in S^1$. Therefore, any point in $S^1\times\{0\}$ is contained in a control set of $\Sigma_{S^1\times\R^2}$. However, by (ii) in Proposition \ref{invariance}, if we consider $u\in\Omega$ with $u\neq 0$ the curve $s\in\R^+\mapsto\varphi_2(s, (t, 0), u)$ always increases its distance from $0\in\R^2$. As a consequence, the same happens for the solution of $\Sigma_{S^1\times\R^2}$ with relation to $S^1\times\{0\}$.  Therefore, two distinct points in $S^1\times\{0\}$ can be in the same control set. Therefore, $\{(t, 0)\}$ is a control set for any $t\in S^1$, concluding the proof.
\end{itemize}

\begin{remark}
The previous result shows that, under the LARC, any LCS on $SE(2)$ whose associated linear vector field $\XC=(A, \xi)$ satisfies $\det A\neq0$ admits a control set $\CC_{SE(2)}$ with nonempty interior satisfying $(0, 0)\in\overline{\CC_{SE(2)}}$. Moreover, $(0, 0)\in\inner\CC_{SE(2)}$ if and only the eigenvalues of $A$ are not real.
\end{remark}

\begin{remark}
The previous result show also a nice bifurcation behavior of the periodic orbit on the boundary of $\CC_{SE(2)}$. In fact, if the matrix $A$ has eigenvalues $\lambda\pm i\mu$, with $\lambda>0$ and $\mu\neq0$, the periodic orbit on the boundary of the open control set $\CC_{SE(2)}$ has period $\frac{2\pi}{\mu}$. Since,
$$\mu\rightarrow 0\;\;\;\implies\;\;\;\left|\frac{2\pi}{\mu}\right|\rightarrow+\infty,$$
we get that as $\mu$ approachs zero  the period of the orbit goes to infinity and such periodic orbit turns into a continuum of one points control sets.
\end{remark}


\begin{thebibliography}{99}
	\bibitem {DSAy1}
	\newblock V. Ayala and A. Da Silva, 
	\newblock \emph{On the characterization of the controllability property for linear control systems on nonnilpotent, solvable three dimensional Lie groups.} 
	\newblock Journal of Differential Equations, 266 No 12 (2019), 8233-8257
	
		\bibitem {DSAy2}
	\newblock V. Ayala and A. Da Silva, 
	\newblock \emph{Control sets of linear control systems on $\R^2$. The complex case.} 
	\newblock https://arxiv.org/abs/2205.01808
 

   	\bibitem{CK}
    \newblock F. Colonius and W. Kliemann, 
    \newblock \emph{The Dynamics of Control},
	\newblock Birkh\"{a}user, Boston, 2000.
	
	   	\bibitem{CK2}
    \newblock F. Colonius, A. J. Santana and J. Setti, 
    \newblock \emph{Control sets for bilinear and affine systems},
	\newblock Mathematics of Control, Signals and Systems, 34, 1-35.
			
	\bibitem{Iv}
	\newblock V. Ivancevic and T. Ivancevic, 
	\newblock \emph{New trends in control theory},
    \newblock World Scientific Publishing, 2013.		
			
	\bibitem{Iv2}
	\newblock V. Ivancevic and T. Ivancevic, 
	\newblock \emph{Symplectic Rotational Geometry in Human Biomechanics},
    \newblock SIAM Rev. 46(3), 455–474, (2004)	

    \bibitem{Jurd}
	\newblock V. Jurdjevic, 
	\newblock \emph{Geometric Control Theory},
    \newblock Cambridge Series in Advances Mathematics 52



\end{thebibliography}
\end{document}